\documentclass[12pt,a4paper,reqno]{amsart}
\usepackage[headings]{fullpage}

\usepackage{amsmath,amssymb,color,hyperref}

\newtheorem{thm}{Theorem}
\newtheorem{mydef}[thm]{Definition}
\newtheorem{prop}[thm]{Proposition}
\newtheorem{lem}[thm]{Lemma}

\renewcommand{\iint}{\int\hspace{-.8em}\int}
\newcommand\RR{\mathbb{R}}
\newcommand{\zzz}{\mathbf{z}}

\usepackage[foot]{amsaddr}

\title[Addendum to ``Local Controllability of the \ldots\dots\ Micro-Swimmer'']{\bf Addendum to ``Local Controllability of the Two-Link Magneto-Elastic Micro-Swimmer''}

\author{Laetitia Giraldi \qquad Pierre Lissy \\ Cl{\'e}ment Moreau \qquad Jean-Baptiste Pomet}

\address{L. Giraldi and J.-B. Pomet are with Universit{\'e} C{\^o}te d'Azur, Inria, CNRS, LJAD, France}
\address{P. Lissy is with CEREMADE, Universit\'{e} Paris-Dauphine,
  Paris, France}
\address{C. Moreau is with ENS de Cachan, France.}
\email{Laetitia.Giraldi@inria.fr}
\email{Lissy@ceremade.dauphine.fr}
\email{Clement.Moreau@ens-cachan.fr}
\email{Jean-Baptiste.Pomet@inria.fr}

\begin{document}

\maketitle

\begin{abstract}
In the above mentioned note
(\href{http://hal.archives-ouvertes.fr/hal-01145537}{\texttt{<hal-01145537>}},
\href{http://arxiv.org/abs/1506.05918}{\texttt{<arXiv:1506.05918>}},
published in IEEE Trans. Autom. Cont., 2017), the first
and fourth authors proved a local controllability result around the
straight configuration for a class of magneto-elastic micro-swimmers.
That result is weaker than the usual small-time local
  controllability (STLC), and the authors left the STLC question open. The present addendum closes it by showing that these systems cannot be STLC.
\end{abstract}


\vspace{3\baselineskip}

\section{Model of the magneto-elastic micro-swimmer}

Keeping the same notations as in \cite{giraldi2015local}, the planar
micro swimmer's dynamics are
given by 
\begin{equation}\label{systeme_controle_robot}
\dot{\zzz} = \mathbf{F}_0(\zzz) + H_{\parallel}\,  \mathbf{F}_1(\zzz)+ H_{\!\bot}\, \mathbf{F}_2(\zzz)
\end{equation}
where (see figure 1 in \cite{giraldi2015local}):
\begin{itemize}
\item the state is $\zzz=(x,y,\theta,\alpha)$ with $\alpha$ an angle
describing the swimmer's shape and $x,y,\theta$ two 
coordinates and an angle describing its position,
\item the control is
$(H_{\!\bot},H_{\parallel})$, the coordinate vector of 
the external magnetic field in a moving frame, the norm on the control space $\RR^2$ being the sup-norm:
\begin{equation*}
  \|(H_{\!\bot},H_{\parallel})\|=\max\{\,|H_{\!\bot}|\,,\,|H_{\parallel}|\,\}\,,
\end{equation*}
\item the $\mathbf{F}_i$'s may be expressed as follows, with  $f_{i,j}$ twelve functions\footnote{The
  notation $f_{i,j}$ is not present in \cite{giraldi2015local}. }
of one variable explicitly derived
from \cite[Prop. II.1 and (12)-(16)]{giraldi2015local}:
\begin{equation}\label{systeme_controle_robot2}
\mathbf{F}_i(\zzz)= 
\begin{pmatrix}  \cos\theta & \sin\theta & 0 & 0 \\
				-\sin\theta & \cos\theta & 0 & 0 \\
                0 & 0 & 1 & 0 \\
                0 & 0 & 0 & 1 
\end{pmatrix}
\!\!
\begin{pmatrix} f_{i,1}(\alpha) \\ f_{i,2}(\alpha) \\ f_{i,3}(\alpha) \\ f_{i,4}(\alpha) \end{pmatrix}.
\end{equation}
\end{itemize}

\bigskip

In \cite{giraldi2015local}, the dynamics, hence the functions $f_{i,j}$, depend on: the
length $\ell_i$ of each segment  ($i=1,2$), its 
magnetization $M_i$, its longitudinal and transversal
hydrodynamic drag constants $\xi_i,\eta_i$, and an elastic constant
$\kappa$. It is assumed that $\kappa>0$ and that, for each $i$,
$\ell_i>0$, $\xi_i>0$, $\eta_i>0$ and $M_i\neq0$.
In this addendum, we further assume that
the two links have the same length and hydrodynamic constants, i.e. we
define:
\begin{align}
  \ell=\ell_1=\ell_2,\;
  \xi=\xi_1=\xi_2,\;
  \label{eq:2}
  \eta=\eta_1=\eta_2. 
\end{align}
This assumption makes the redaction easier to follow but it
does not alter the nature of the proofs.

The equilibria of interest are
$\left((x^e,y^e,\theta^e,0),(0,0)\right)$ in the state-control
space, with $(x^e,y^e,\theta^e)$ arbitrary in $\mathbb R^2\!\times\! [0,2\pi]$.
Using invariance by translation and rotation \cite{giraldi2015local},
one may, without loss of
generality, suppose $(x_e,y_e,\theta_e)=(0,0,0)$ and consider only the equilibrium
$\text{\textbf{O}}=\bigl((0,0,0,0),(0,0)\bigr)$.

\section{Some local controllability concepts}
Consider a smooth continuous-time control system
\begin{equation}
  \label{eq:1}
  \dot{z} = f(z,u)
\end{equation}
with state $z$ in $\mathbb{R}^n$ and control $u$ in $\mathbb{R}^m$. We
endow $\RR^m$ with a norm $\|\cdot\|$ and always assume that 
$u$ is essentially bounded.
Let $(z_e,u_e)$ be an equilibrium of \eqref{eq:1}, i.e. $f (z_e,u_e)=0$.
 
The following definition introduces an ad hoc notion of controllability for the sake of clarity.
\begin{mydef}[STLC($q$)]
\label{def-stlcq}
Let $q$ be a non-negative number.
The control system \eqref{eq:1} is \emph{STLC($q$) at $(z_e,u_e)$} 
if and only if, for every $\varepsilon >0$, there exists $\eta >0$ such that, for every $z_0,z_1$ in the ball centered at $z_e$ with radius $\eta$,
there exists a solution $(z(\cdot),u(\cdot)):[0,\varepsilon]\to\mathbb{R}^{n+m}$ of \eqref{eq:1} 
such that  $z(0)=z_0$, $z(\varepsilon)=z_1$, and, for almost all $t$ in $[0,\varepsilon]$,
\begin{equation*}
\| u(t)-u_e \| \leqslant q+\varepsilon\,.
\end{equation*}
\end{mydef}

Let us also recall the classical definition of STLC. 
\begin{mydef}[STLC]
\label{def-stlc}
The system \eqref{eq:1} is \emph{STLC (small-time locally controllable) at $(z_e,u_e)$} 
if and only if it is STLC(0) at $(z_e,u_e)$. 
\end{mydef}

The following necessary condition for STLC will be used. 
\begin{lem}[Loop trajectories]
\label{prop-loops}
If \eqref{eq:1} is STLC at $(z_e,u_e)$, then, for any
$\varepsilon>0$,  there exists a solution $t\mapsto(z^\varepsilon(t),u^\varepsilon(t))$ of \eqref{eq:1},
defined for $t$ in $[0,\varepsilon]$, such that 
\begin{itemize}
\item $z^\varepsilon(0)=z^\varepsilon(\varepsilon)=z_e$,
\item $z^\varepsilon(t)\neq z_e$ for at least one $t$ in $[0,\varepsilon]$,
\item $\|u^\varepsilon(t)-u_e\|\leqslant\varepsilon$ for almost all $t$ in $[0,\varepsilon]$.
\end{itemize}
\end{lem}
\begin{proof}
Let $\varepsilon>0$. There exists $\eta >0$ such that, for every $z_\star$ in the ball centered at $z_e$ with radius $\eta$,
there is a solution $(z^\varepsilon(\cdot),u^\varepsilon(\cdot)):[0,\varepsilon/2]\to\mathbb{R}^{n+m}$ of \eqref{eq:1} 
such that  $z^\varepsilon(0)=z_e$,
$z^\varepsilon(\varepsilon/2)=z_\star$, and  $\| u^\varepsilon(t)-u_e
\| \leqslant \varepsilon/2$  for almost all $t$.
Pick one such $z_\star$ diffwrent from $z_e$. System \eqref{eq:1} being autonomous, there also exists a solution 
$(z^\varepsilon(\cdot),u^\varepsilon(\cdot)):[\varepsilon/2,\varepsilon]\to\mathbb{R}^{n+m}$ of \eqref{eq:1} 
such that  $z(\varepsilon/2)=z_\star$, $z(\varepsilon)=z_e$, and, for almost all $t$ in $[\varepsilon/2,\varepsilon]$,  $\| u(t)-u_e \| \leqslant \varepsilon/2$.
Then, $(z^\varepsilon(\cdot),u^\varepsilon(\cdot)) :[0,\varepsilon]\to\mathbb{R}^{n+m}$ verifies all the desired properties.
 \end{proof}


\section{Complements to the original note}

The following proposition reformulates the results from
\cite{giraldi2015local}. Without assumption \eqref{eq:2},
$\xi\!\neq\!\eta$ would be replaced by $(\xi_1,\xi_2)\!\neq\!(\eta_1,\eta_2)$ and $M_1\!\neq\! M_2$ by \cite[eqn. (20)]{giraldi2015local}.

\begin{prop}[\cite{giraldi2015local}, Thm. III.4 and Prop. III.1] Assume \eqref{eq:2}. 
\label{th-old}
The control system \eqref{systeme_controle_robot} is 
STLC$\displaystyle\left( 2 \kappa \,\bigl.\left| M_1+M_2\right|\bigr/\left|M_1 M_2\right|\right)$ at \textbf{O}
if $\xi\neq\eta$ and $M_1\neq M_2$.
Otherwise, it is not STLC$(q)$ for any $q\geq0$.
\end{prop}



Unless $M_1+M_2=0$, 
STLC$\displaystyle\left( 2 \kappa \,\bigl.\left|M_1+M_2\right|\bigr/\left|M_1 M_2\right|\right)$ 
does not imply STLC. 
The purpose of the present addendum is to prove the following result:


\begin{thm}
\label{th-main}
Assume \eqref{eq:2}. \\ If $\xi\neq\eta$, $M_1\neq M_2$ and $M_1+M_2\neq0$,
system \eqref{systeme_controle_robot} is not STLC at \textbf{O}.
\end{thm}


\begin{proof} 
From \cite[Prop. II.1 and (12)-(16)]{giraldi2015local}, one readily verifies that the functions $f_{i,j}$ introduced in \eqref{systeme_controle_robot2} have the
following expansions around $\alpha=0$:%
\begin{align}
\nonumber
&f_{2,j}(\alpha)=\mathcal{O}(\alpha), && \hspace{-8em}j\in\{1,2,3,4\},  
\\
\nonumber
&
f_{0,1}(\alpha)=a_1\hspace{.05em}\alpha^2+\!\mathcal{O}(\alpha^{\!3}), &&
f_{1,1}(\alpha)=b_1\alpha+\!\mathcal{O}(\alpha^{\!2}),
\\
\label{dev-4} 
&
f_{0,2}(\alpha)=\mathcal{O}(\alpha^{\!2}), &&
f_{1,2}(\alpha)=b_2+\!\mathcal{O}(\alpha),
\\
\nonumber
&f_{0,3}(\alpha)=\frac{a}2\alpha+\!\mathcal{O}(\alpha^{\!2}),  && 
f_{1,3}(\alpha)=b_3+\!\mathcal{O}(\alpha),
\\
\nonumber 
&f_{0,4}(\alpha)=-a\,\alpha+\!\mathcal{O}(\alpha^{\!2}), &&
  f_{1,4}(\alpha)=b_4+\!\mathcal{O}(\alpha),
\end{align}
with
\begin{align}
\nonumber
&\!
a_1\!=\!\frac {3\kappa}{\ell^2\eta},\,b_1\!=\!\frac32  \frac {M_2\!-\!M_1}{\ell^2\eta}-\frac38\frac {M_1\!+\!M_2}{
\ell^2\xi}
,\,
  b_2 =\ \frac34 \frac{M_1\!+\!M_2}{\ell^2 \eta},
\\
\label{a1a2b1b2b3b4}
&\!
  a_2=\frac{24\,\kappa}{\ell^3\eta},\; 
  b_3= \frac{3(5 M_2 -3M_1)}{2\,\ell^3 \eta}, \;
  b_4= \frac{12 (M_1-M_2)}{\ell^3 \eta}.
\end{align}
The assumptions before \eqref{eq:2} and these of the theorem imply
$b_4\neq0$, $a_2\neq0$, $M_1+M_2\not =0$ and $M_1-M_2\not =0$, hence
\begin{equation}
z_4=\frac1{b_4}\,\alpha,\;
z_3=\frac{8\,(M_1-M_2)}{a_2\,(M_1+M_2)}\Bigl(b_4\,\theta-b_3\,\alpha\Bigr)
\label{eq_brunovsky222}
\end{equation}
defines a change of coordinates\footnote{For the reader's information: the linear approximation of \eqref{systeme_controle_robot} is in (non-controllable) Brunovsky form in coordinates $(x,\,y-b_2z_3\,,z_3,z_4)$.} $(x,y,\theta,\alpha)\mapsto(x,y,z_3,z_4)$. 
Since 
$8(M_1\!-\!M_2)/(M_1\!+\!M_2)=1/(1/2\!+\!b_3/b_4)$, one deduces from
\eqref{systeme_controle_robot}, \eqref{systeme_controle_robot2},
 \eqref{dev-4},   and \eqref{eq_brunovsky222} the
following expressions of $\dot{z}_3$ and $\dot{z}_4$, where $r_{i,j}$
($i=0,1,2$, $j=3,4$) are smooth functions of one variable:
\begin{equation}
\!\!\!
\begin{array}{l}
\dot{z}_3 \! = z_4\left(\,1+z_4\,r_{0,3}(\!z_4\!) + H_{\!\bot}\, r_{1,3}(\!z_4\!) + H_{\parallel}\, r_{2,3}(\!z_4\!)\right) ,\\
\dot{z}_4 \! = H_{\!\bot} \!- z_4\!\left(a_2+z_4\,  r_{0,4}(\!z_4\!) + H_{\!\bot} \, r_{1,4}(\!z_4\!) + H_{\parallel} \, r_{2,4}(\!z_4\!)\right) \!.
\end{array}
\!\!\!\!\!\! \!\!\!\!\!\! \!\!\!\!\!\! 
\label{eq_brunovsky2}
\end{equation}

Substituting
$\alpha=b_4\,z_4$ and $\displaystyle\theta=b_3\,z_3+\frac{a_2\,(M_1+M_2)}{8b_4\,(M_1-M_2)}\,z_4$
in \eqref{systeme_controle_robot2}, expanding $\sin(\theta)$ and $\cos(\theta)$ around $0$ and using 
 \eqref{dev-4} one gets, with $c_1,c_2,c_3$ three constants that may
easily be computed from $a_1,a_2,b_1,b_2,b_3,b_4$, the expression:
\begin{gather}
\label{eq:102}
\!\!\!\! \!\!\!
\dot x = c_3\,z_4^{\,2}+ (c_1 z_3\!+\!c_2 z_4)H_{\!\bot}+z_4^{\,2}R_1+z_3z_4\,R_2+z_3^{\,2}R_3
\!\!\!\!
\\[.7ex]
\begin{split}
\label{eq:102b}
\!\!\!\! \!
\text{with}\ R_1=\,&z_4\, \rho_{1}(z_3,z_4)+\rho_{2}(z_3,z_4)\,H_{\!\bot}+\rho_{3}(z_3,z_4)\,H_{\parallel},
\!\!
\\
R_2=\,&\; \rho_{4}(z_3,z_4)+\rho_{5}(z_3,z_4)\,H_{\!\bot}+\rho_{6}(z_3,z_4)\,H_{\parallel},
\!\!
\\
R_3=\,&\;  \rho_{7}(z_3,z_4)+\rho_{8}(z_3,z_4)\,H_{\!\bot}+\rho_{9}(z_3,z_4)\,H_{\parallel}, 
\!\!
\end{split}
\end{gather}
and $\rho_i$, $i=1\ldots9$, nine smooth functions of two variables.
Then, defining
$
\zeta = x - c_1 z_3 z_4 - \frac{1}{2} c_2 z_4^{\,2}
$,
one has
\begin{equation}
\label{zetadot}
\begin{split}
  &\dot{\zeta} =z_4^{\,2}\bigl(c_0+ \widetilde{R}_1\bigr) +z_3
  z_4\,\widetilde{R}_2+z_3^{\,2}\,\,\widetilde{R}_3
\\&\text{with}\ c_0=c_3 +a_2 c_2 - c_1
\end{split}
\end{equation}
with 
$\widetilde{R}_1,\widetilde{R}_2,\widetilde{R}_3$ three functions of
$z_3$, $z_4$, $H_{\!\bot}$, $\!H_{\parallel}$ that can be expended similarly to $R_1$, $R_2$ and $R_3$ in \eqref{eq:102b}. 
Computing $c_0$ from the expressions of $c_1,c_2,c_3$, one finds that it is nonzero from the assumptions of Theorem \ref{th-main}:
\begin{equation}
\label{eq:103}
c_0 =\frac{108 \kappa}{\ell^8 \eta^3 \xi} (M_2^2-M_1^2) (\eta-\xi)\neq0.
\end{equation}

From Lemma \ref{prop-loops}, for each $\varepsilon>0$, there exists a ``loop''
$$
t\mapsto(x^\varepsilon(t),y^\varepsilon(t),\theta^\varepsilon(t),\alpha^\varepsilon(t),H_{\!\bot}^\varepsilon(t),\!H_{\parallel}^\varepsilon(t)\,)
$$
defined on $[0,\varepsilon]$, solution of \eqref{systeme_controle_robot}, and such that 
\begin{gather}
\label{loop-eps}
|H_{\!\bot}^\varepsilon(t)|\leqslant\varepsilon\;\text{and}\;|H_{\parallel}^\varepsilon(t)|\leqslant\varepsilon
\ \text{for all $t$ in $[0,\varepsilon]$,}
\\
\label{loop1loop2}
\begin{split}
  &(x^\varepsilon(0),y^\varepsilon(0),z_3^\varepsilon(0),z_4^\varepsilon(0))=(0,\!0,\!0,\!0),
  \\
  &(x^\varepsilon(\varepsilon),y^\varepsilon(\varepsilon),z_3^\varepsilon(\varepsilon),z_4^\varepsilon(\varepsilon))=(0,\!0,\!0,\!0),
\end{split}
\\
\label{trueloop}
\!\!\!\! \!\!\!
(x^\varepsilon(t),y^\varepsilon(t),z_3^\varepsilon(t),z_4^\varepsilon(t))\neq(0,\!0,\!0,\!0)
\text{ for one $t$ in $[0,\varepsilon]$,}\!\!\!
\end{gather}
where $z_3^\varepsilon(t),z_4^\varepsilon(t)$ are defined from $(\theta^\varepsilon(t),\alpha^\varepsilon(t))$ as in \eqref{eq_brunovsky222}.
Along these solutions, the functions $r_{i,j}$ and $\rho_i$ are
bounded uniformly with respect to $t$ in $[0,\varepsilon]$ and
$\varepsilon$ in $(0,\varepsilon_0]$ for some small enough $\varepsilon_0>0$. 
In particular, using \eqref{eq_brunovsky222} and \eqref{zetadot}, we deduce that for each $\varepsilon\in (0,\varepsilon_0]$, there are three functions $u^\varepsilon(.)$, $v^\varepsilon(.)$, $w^\varepsilon(.)$
such that
\begin{align}
\label{eq_brunovsky3z3}
&\dot{z}^\varepsilon_3(t)  = z_4^\varepsilon(t)\,u^\varepsilon(t)\,,
\\
\label{eq_brunovsky3zeta}
&\dot{\zeta}^\varepsilon(t) =z_4^\varepsilon(t)^2 \bigl(c_0+\varepsilon\,w^\varepsilon(t)\bigr)
+z_3^\varepsilon(t) z_4^\varepsilon(t)\;v^\varepsilon(t) + z_3^\varepsilon(t)^2 s^\varepsilon(t)\,, 
\\[.3ex]
\label{bru4}
&|u^\varepsilon(t)|\leqslant K,\;|v^\varepsilon(t)|\leqslant K,\;|w^\varepsilon(t)|\leqslant K,\;|s^\varepsilon(t)|\leqslant K\,.
\end{align}
 Here and hereafter, $K>0$ denotes a constant independent of $\varepsilon$ and $t$ that may vary from line to line.
One has:
\begin{lem}
\label{lem}
Equations \eqref{eq_brunovsky3z3} and \eqref{bru4} imply, for
$\varepsilon$ in $(0,\varepsilon_0]$,
\begin{displaymath}
  \int_0^\varepsilon \hspace{-.02em} |z_3^\varepsilon(t) z_4^\varepsilon(t)|\mathrm{d}t
\leqslant
K\varepsilon \int_0^\varepsilon \hspace{-.02em}
z_4^\varepsilon(t)^2\mathrm{d}t
\ \ \text{and}\ \ 
\int_0^\varepsilon \hspace{-.02em} z_3^\varepsilon(t)^2\mathrm{d}t
\leqslant
K^2\varepsilon^2 \! \int_0^\varepsilon \hspace{-.02em}  z_4^\varepsilon(t)^2\mathrm{d}t\,.
\end{displaymath}
\end{lem}

\smallskip

Let us temporarily admit this lemma. Then, equations \eqref{loop1loop2} imply
$\int_0^\varepsilon\dot\zeta^\varepsilon(t)\mathrm{d}t=0$. Substituting
$\dot\zeta^\varepsilon(t)$ from
\eqref{eq_brunovsky3zeta} and using 
\eqref{bru4} and Lemma~\ref{lem} yields, for any
$\varepsilon$ in $(0,\varepsilon_0]$,
\begin{equation}
\label{eq:111}
|c_0|\int_0^\varepsilon z_4^\varepsilon(t)^2\mathrm{d}t
\leqslant
\left(2K\varepsilon+K^2\varepsilon^2\right)\int_0^\varepsilon z_4^\varepsilon(t)^2\mathrm{d}t\,.
\end{equation}
Since $c_0\neq0$, this implies that $z_4^\varepsilon(t)$ is
identically zero on $[0,\varepsilon]$ for  $\varepsilon>0$ small enough. 
From \eqref{eq_brunovsky2}, this implies that the control
$H _\bot^\varepsilon(t)$ is identically zero. Since all the maps
$f_{i,j}$ with $i\neq1$ are zero at zero (see
\eqref{systeme_controle_robot2} and \eqref{dev-4}),
all state variables are constant if $z_4^\varepsilon$ and $H_\bot^\varepsilon$ are identically
zero, meaning that $(x^\varepsilon(t),y^\varepsilon(t),\theta^\varepsilon(t))$ are identically zero on $[0,\varepsilon]$ for  $\varepsilon>0$ small enough. Therefore, any small enough loop with small enough control
is trivial, which contradicts \eqref{trueloop} and hence contradicts STLC.
\end{proof}
\begin{proof}[Proof of Lemma \ref{lem}]
From \eqref{eq_brunovsky3z3} and \eqref{bru4},
one gets
\begin{displaymath}
|z_3^\varepsilon(t)|\leqslant
K\!\!\int_0^t\hspace{-.3em}|z_4^\varepsilon(\tau)|\,\mathrm{d}\tau,\ \ 
\int_0^\varepsilon \hspace{-.3em}|z_3^\varepsilon(t)|\mathrm{d}t \leqslant K\varepsilon\!\!\int_{0}^{\varepsilon}\hspace{-.3em}| z_4^\varepsilon(\tau)|\mathrm{d}\tau\,.
\end{displaymath}
The following two inequalities follow:
\begin{align*}
\nonumber
\int_0^\varepsilon \!\!|z_3^\varepsilon(t) z_4^\varepsilon(t)|\,\mathrm{d}t\;
&\leqslant
K\!
\int_0^\varepsilon\Bigl(\int_0^t |z_4^\varepsilon(\tau)|\, \mathrm{d}\tau\Bigl)   |z_4^\varepsilon(t)|\,\mathrm{d}t
\\[-.3ex]
&
  \leqslant
K \!\iint_{(t,\tau)\in[0,\varepsilon]^2}
\hspace{-2.6em}
|z_4^\varepsilon(t)|\,|z_4^\varepsilon(\tau)|\mathrm{d}t\,\mathrm{d}\tau
=
K\,\bigl(\int_0^\varepsilon\hspace{-.6em}
  |z_4^\varepsilon(t)|\mathrm{d}t\,\bigr)^2\,,
\end{align*}
\begin{align*}
\int_0^\varepsilon\!\!z_3^\varepsilon(t)^2\mathrm{d}t\;
&\leqslant
K\! \int_0^\varepsilon\Bigl(\int_0^t |z_4^\varepsilon(\tau)|\, \mathrm{d}\tau\Bigl)   |z_3^\varepsilon(t)|\,\mathrm{d}t
  \leqslant
K \!\iint_{(t,\tau)\in[0,\varepsilon]^2}
\hspace{-2.6em}
|z_3^\varepsilon(t)|\,|z_4^\varepsilon(\tau)|\mathrm{d}t\,\mathrm{d}\tau
\\[-.2ex]
&\hspace{2.5em}
=K\bigl(\int_0^\varepsilon\hspace{-.6em} |z_3^\varepsilon(t)|\mathrm{d}t\bigr)
\bigl(\int_0^\varepsilon\hspace{-.6em} |z_4^\varepsilon(t)|\mathrm{d}t\bigr)
\leqslant K^2\varepsilon\,\bigl(\int_0^\varepsilon\hspace{-.6em} |z_4^\varepsilon(t)|\mathrm{d}t\bigr)^2\,.
\end{align*}
We conclude by applying the Cauchy-Schwartz inequality. 
\end{proof}

\section{Conclusion}

We proved that the local controllability results in
\cite{giraldi2015local} are sharp in the sense that STLC
occurs only for the values of the parameters for which it was already proved in that note.

On the one hand, from the theoretical point of view of controllability,
although it deals with a
very specific class of systems, Theorem~\ref{th-main} is a necessary
condition for STLC.
Conditions for STLC have been much studied in the last
decades, see for instance \cite{coron2007control} or \cite{Kaws90dekker} and references
therein.
Many sophisticated and powerful sufficient conditions have been
stated, but necessary conditions are always specific, see for instance \cite{Kaws88contemp,Kras98}.
Theorem~\ref{th-main} is not, to the best of our knowledge, a
consequence of known necessary conditions.

On the other hand, the implications for locomotion at low Reynolds number via an
external magnetic field are not clear. Comments on that matter are
left to further research.

\end{document}